\documentclass{amsart}
\usepackage{amssymb}
\usepackage{graphicx}
\usepackage{algorithm}
\usepackage{algpseudocode}
\usepackage{hyperref}

\newtheorem{theorem}{Theorem}[section]
\newtheorem{lemma}[theorem]{Lemma}

\theoremstyle{definition}
\newtheorem{definition}[theorem]{Definition}

\theoremstyle{remark}
\newtheorem{remark}[theorem]{Remark}

\numberwithin{equation}{section}

\newcommand{\N}{\ensuremath{\mathbb{N}}}

\begin{document}

\title[Perfect colorings of regular graphs]{Perfect colorings of regular graphs}

\author{Joseph Ray Clarence Damasco}
\address{Institute of Mathematics, University of the Philippines Diliman, Quezon City, Philippines}
\thanks{JD is grateful to the University of the Philippines 
System for financial support through its Faculty, REPS, and Administrative Staff 
Development Program.}

\author{Dirk Frettl\"oh}
\address{Bielefeld University, Postfach 100131, 33501 Bielefeld, Germany}
\email{dirk.frettloeh@udo.edu}
\urladdr{https://math.uni-bielefeld.de/\textasciitilde frettloe}
\thanks{DF thanks the Research Center of Mathematical
Modeling (RCM$^2$) at Bielefeld University for financial support.
We are grateful to Ferdinand Ihringer and to Michael Stiebitz for valuable
information, and to the two anonymous referees for valuable remarks. 
Special thanks to Caya and Lasse Schubert for providing some
of the 3- and 4-colorings of the cube and the dodecahedron.}

\subjclass[2010]{Primary 05C15, 05C50}

\date{}

\begin{abstract}
A vertex coloring of some graph is called \emph{perfect} if each vertex 
of color $i$ has exactly $a_{ij}$ neighbors of color $j$. Being perfect
imposes several restrictions on the color adjacency matrix $(a_{ij})$. 
We give a characterization of color adjacency matrices of perfect colorings
of graphs, and in particular, connected graphs. Using this result we 
determine the lists of all color adjacency matrices corresponding to 
perfect colorings of 3-regular, 4-regular and 5-regular graphs with 
two, three and four colors. Finally, using these lists, we 
determine all perfect colorings of the edge graphs of the Platonic solids with 
two, three and four colors, respectively. 
\end{abstract}

\maketitle

\section{Introduction}

Perfect colorings of graphs and related concepts have been
studied in several contexts: algebraic graph theory, combinatorial designs,
coding theory, finite geometry; and under several different names:
equitable partitions, completely regular vertex sets, distance partitions,
association schemes, etc. Some connections between these contexts are stated 
in Remark \ref{rem:terms} below. For a broader overview see \cite{FI,I}.

Throughout the paper let $G=(V,E)$ be a finite, undirected, simple, loop-free graph.
A partition of $V$ into disjoint nonempty sets $V_1, \ldots, V_m$ is called 
an \emph{$m$-coloring} of $G$. 
Note that we do not require adjacent vertices to have different colors. 
\begin{definition}
A coloring of the vertex set $V$ of some graph $G=(V,E)$ with $m$ colors
is called \emph{perfect} if (1) all colors are used, and (2) for all $i,j$ 
the number of neighbors of colors $j$ of any vertex $v$ of color $i$ is 
a constant $a_{ij}$. The matrix $A=(a_{ij})_{1 \le i,j \le m}$ is called the
\emph{color adjacency matrix} of the perfect coloring.
\end{definition}
See Figures \ref{fig:tetrahedron}-\ref{fig:icosahedron} in Appendix B for some examples
of perfect colorings. Note that for $m=|V|$ the color adjacency matrix equals
the adjacency matrix of $G$.

\begin{remark}
  In some sources (e.g.\ \cite[Sec.~9.3]{GR} and \cite{BFW}) perfect colorings are called
  \emph{equitable partitions}. However, it seems that the term ``equitable 
partition'' is used for two different concepts in graph theory: one is
what we call perfect coloring above, the second is a coloring where
every pair of adjacent vertices has different colors, and where the
number of elements of any two color classes differs by at most one.
See for instance \cite{HS}, or \cite{Fu} and references therein. Hence we will
refer to the first concept by the term perfect coloring here. 
\end{remark}

\begin{remark} \label{rem:terms}
  Perfect colorings --- or very related concepts --- have been studied
  in several contexts. The first source we know about is \cite{Sa}, where
  the color adjacency matrix was introduced to study spectral properties of
  certain graphs. In particular, the following result was shown in 
\cite[Theorem 4.5]{CDS}, or \cite[Theorem 9.3.3]{GR}.
  \begin{theorem} \label{thm:eigenv}
    Let $M$ be the adjacency matrix of some graph $G$ and let $A$ be the
    color adjacency matrix of some perfect coloring of $G$. Then the
    characteristic polynomial of $A$ divides the characteristic polynomial
    of $M$. In particular, each eigenvalue of $A$ is an eigenvalue of $M$
    (with multiplicities).
    \end{theorem}
  To name just a few more examples:
  Any subgroup   of the automorphism group of a graph $G$ induces a perfect 
coloring of $G$ by considering the orbits of the group \cite[Sec.~9.3]{GR}.
However, not every perfect coloring arises from a graph automorphism.
As another example, each
  \emph{distance partition} (coloring the vertices w.r.t.\ their distance
  to some fixed vertex) of a distance regular graph \cite{G} yields a perfect coloring.
For more related work see \cite{FI,I,K}.

  Some concrete perfect colorings for small graphs were constructed 
for instance in \cite{AM,AAb,AvMo,F,GG,M}. Here we generalize several
results from these papers. 
\end{remark}

This paper is organized as follows: 
In Section \ref{sec:charmat}, we give necessary and sufficient conditions for a
matrix to be a color adjacency matrix of a perfect coloring of some graph, 
and in particular, some connected graph. In Section \ref{sec:vercount}, we relate 
the cardinalities of the color classes of perfect colorings with two, three,
and four colors to the entries of color adjacency matrices. 
Using these results we compute in Section \ref{sec:enum} the lists of
all color adjacency matrices of perfect 
colorings with two, three and four colors for $k$-regular connected
graphs for $k \in \{3,4,5\}$, respectively, up to equivalence
by permutations of colors.
To our best knowledge the lists for three and 
four colors have not been published before.

The computations were carried out both in \texttt{sagemath} and in \texttt{scilab}.
The implementation is described in Section \ref{sec:implement}.

As an application we determine in Section 
\ref{sec:platonic} all perfect colorings of the edge graphs
of the Platonic solids using two, three and four colors, respectively.
All perfect 2-colorings of the edge graphs of the Platonic 
solids have been determined in \cite{AK} already. 
The perfect 3-colorings of the edge graphs of the Platonic solids were 
studied in \cite{AA}, but some cases were missed in the preprint version. 
To our knowledge the perfect 4-colorings of the 
edge graphs of the Platonic solids given in this paper are new.

\section{Characterization of color adjacency matrices}\label{sec:charmat}

We will find necessary and sufficient conditions under which a given 
nonnegative integer matrix corresponds to a perfect coloring of a graph
as it seems that the relevant literature on perfect colorings of graphs lack
such explicitly stated conditions. In the ensuing discussion, given a perfect 
coloring for a graph~$G$ with color adjacency matrix 
$A=(a_{ij}) \in \N^{m \times m}$, let $v_{i}$ denote the number of vertices 
in the color class $V_{i}$, for $1 \leq i \leq m$. 

\begin{theorem}\label{thm:charmat}
Suppose $A=(a_{ij}) \in \N^{m \times m}$. 
Then $A$ is a color adjacency matrix for a perfect $m$-coloring of some graph 
$G=(V,E)$ if and only if the following hold:
\begin{enumerate}
\item \textnormal{(}Weak symmetry\textnormal{)} For $1 \leq i, j \leq m$, $a_{ij}=0$ 
if and only if $a_{ji}=0$.
\item \textnormal{(}Consistency\textnormal{)} For any nontrivial cycle 
$(n_{1} \, n_{2} \, \ldots \, n_{t})$ in the symmetric group $S_{m}$ on the 
set $\{1,2,\ldots,m\}$,
\[a_{n_{1},n_{2}}a_{n_{2},n_{3}}\cdots a_{n_{t-1},n_{t}} a_{n_{t},n_{1}} = 
a_{n_{2},n_{1}}a_{n_{3},n_{2}}\cdots a_{n_{t},n_{t-1}} a_{n_{1},n_{t}}.\]
\end{enumerate}
Moreover, there is a connected graph $G$ with a perfect coloring 
corresponding to $A$ if and only if $A$ fulfills (1) and (2), and $A$ is 
irreducible. 
\end{theorem}
A symmetric matrix $M$ is called irreducible if it is not conjugate via a permutation 
matrix to a block diagonal matrix having more than one block.
(By ``block diagonal matrix'' we mean a square matrix having square matrices
on its main diagonal, and all other entries being zero.) It is well-known 
that a directed graph $G$ is connected if and only if its adjacency matrix is
irreducible. A weaker statement is true here: if a graph $G$ is connected 
then its color adjacency matrix is irreducible. (Because one can travel
from any color to any other color.) 

Before proving Theorem~\ref{thm:charmat}, we note that in any perfect coloring,
the subgraph of $G$ induced by the vertices of color $i$ is an $a_{ii}$-regular graph.
In addition, the edges between the vertices of color $i$ and those of color $j$
form the edge set of an $(a_{ij},a_{ji})$-biregular graph, where by a 
$(p,q)$-biregular graph we mean a bipartite graph with bipartition 
$(U,W)$ such that each vertex in $U$ has degree $p$ and each vertex in $W$ 
has degree $q$. 

We shall prove the sufficiency of conditions (1) and (2) of 
Theorem~\ref{thm:charmat} constructively using Lemmas~\ref{lem:reg} 
and~\ref{lem:breg}, which characterize regular and
biregular graphs. All variables in these statements are
nonnegative integers. We include a proof of Lemma~\ref{lem:breg}
because we are not aware of a reference containing
a proof of it.

\begin{lemma}[\cite{ChLZ}]\label{lem:reg}
There exists a $k$-regular graph with $n$ vertices if and only if $n \geq k+1$ 
and $nk$ is even.
\end{lemma}

Lemma \ref{lem:reg} is a simple consequence of the Erd\H{o}s-Gallai Theorem \cite{EG}.

\begin{lemma}\label{lem:breg}
There exists a $(p,q)$-biregular 
graph with bipartition $(U,W)$ and $|U| = r$, $|W| = s$
if and only if $p \le s$, $q \le r$, and $pr=qs$.
\end{lemma}

\begin{proof}
That $p \le s$, $q \le r$, and $pr=qs$ are necessary follows from the
definition of biregular graphs. To prove the converse,
we assume $p$, $q \neq 0$ and construct a graph with the desired properties. 
Denote the vertices in $U$ by $u_0$, $u_{1}$, $\ldots$, $u_{r-1}$ and the vertices in 
$W$ by $w_0$, $w_{1}$, $\ldots$, $w_{s-1}$. We will use a greedy construction:
join vertex $u_0$ with $w_0$, $w_{1}$, $\ldots$, $w_{p-1 \bmod s}$, join vertex 
$u_{1}$ with $w_{p \bmod s}$, $w_{p+1 \bmod s}$, $\ldots$, $w_{2p-1 \bmod s}$, 
and so on. That is, for each $0 \leq a \leq r-1$, join vertex $u_{a}$ to
$w_{ap+b \bmod s}$ for all $0 \leq b \leq p-1$. Because $pr=qs$, every vertex in 
$U$ is adjacent to $p$ vertices in $W$, and every vertex in $W$ is adjacent to 
exactly $q$ vertices in $U$. 
\end{proof}

We now prove our characterization of color adjacency matrices.

\begin{proof}[Proof of Theorem \ref{thm:charmat}]

The necessity of the first condition follows trivially from the symmetry
of the vertex adjacency relation. Meanwhile, the necessity of the second 
condition when $t=2$ is also trivial: 
$a_{n_{1},n_{2}}a_{n_{2},n_{1}} = a_{n_{2},n_{1}}a_{n_{1},n_{2}}$. 

When $2 < t \leq m$, the identity arises from the fact that for any $i$ and $j$,
$a_{ij}v_{i} = a_{ji}v_{j}$. This is true by counting the number of edges between
$V_{i}$ and $V_{j}$ in two ways. Thus, $a_{n_{1},n_{2}}v_{n_{1}}=a_{n_{2},n_{1}}v_{n_{2}}$, and 
therefore $a_{n_{1},n_{2}}a_{n_{2},n_{3}}v_{n_{1}}=a_{n_{2},n_{1}}a_{n_{2},n_{3}}v_{n_{2}}
=a_{n_{2},n_{1}}a_{n_{3},n_{2}}v_{n_{3}}$. It follows by induction that 
\begin{equation} \label{eq:an1an2}
a_{n_{1},n_{2}}a_{n_{2},n_{3}}\cdots a_{n_{t-1},n_{t}} v_{n_{1}} = 
a_{n_{2},n_{1}}a_{n_{3},n_{2}}\cdots a_{n_{t},n_{t-1}} v_{n_{t}}.
\end{equation}
Combining this with $a_{n_{1},n_{t}}v_{n_{1}} = a_{n_{t},n_{1}}v_{n_{t}}$ gives
the equation in condition (2).

Before establishing the converse, we note the equation arising from the induction 
in the following remark as it will be used later.

\begin{remark}\label{rem:eqcon}\ 
\begin{enumerate}
\item For any nontrivial cycle 
$(n_{1} \, n_{2} \, \ldots \, n_{t})$ in the symmetric group $S_{m}$ on the set $\{1,2,\ldots,m\}$,
\[ a_{n_{1},n_{2}}a_{n_{2},n_{3}}\cdots a_{n_{t-1},n_{t}} v_{n_{1}} = 
a_{n_{2},n_{1}}a_{n_{3},n_{2}}\cdots a_{n_{t},n_{t-1}} v_{n_{t}}.\]
\item The consistency condition means that for any $i$ and $j$, the different 
ways of relating $v_{i}$ and $v_{j}$ by products of the $a_{n_r,n_s}$'s must all agree.
\end{enumerate}
\end{remark}

To prove the converse, given a matrix $A \in \N^{m \times m}$ satisfying conditions
(1) and (2), we construct a graph $G$ with color adjacency matrix $A$,
where the color classes are denoted by $V_{1}, V_{2}, \ldots, V_m$. By possibly conjugating
via a permutation matrix, we assume that $A$ is written as a block diagonal matrix
with the largest number of blocks possible. (For instance, a block consisting of
a diagonal matrix is interpreted as many blocks of size $1$.)

Suppose first that there is only one block. This means that the matrix is not
permutation-conjugate to a block diagonal matrix with more than one block,
and that there is a path from any color to any other color. Hence, 
plugging the nonzero nondiagonal entries into Equation \eqref{eq:an1an2}  
we obtain the ratio of $v_{i}$ to $v_{j}$ for \emph{every} $i \neq j$
There may be several ways of relating $v_{i}$ and $v_{j}$ by entries of $A$, 
but because of condition (2), we know all these ratios are consistent, and 
there is an ordered $m$-tuple of positive integers $(v_{1}',v_{2}',\ldots,v_{m}')$ 
satisfying all required relations.

Moreover, there is a large enough multiple $(v_{1},v_{2},\ldots,v_{m})$ of the $m$-tuple
above such that for each $i$, $v_{i} \geq  a_{ii}+1$, $a_{ii}v_{i}$ is even, 
and $v_{i} \geq a_{ji}$ for $j \neq i$. Let $V_{i}$ have $v_{i}$ elements. 
By Lemmas~\ref{lem:reg} and~\ref{lem:breg}, we may form an
$a_{ii}$-regular graph using the vertices in $V_{i}$, and the edge set of 
an $(a_{ij},a_{ji})$-biregular bipartite graph between distinct cells 
$V_{i}$ and $V_{j}$. The resulting graph has the 
perfect coloring $(V_{1},\ldots,V_{m})$ with color adjacency matrix $A$.
This graph may still be disconnected, but any component of this graph
satisfies the adjacency relations described by $A$. Thus, we choose $G$
to be one of the components of that auxiliary graph. Then $G$ is a
connected graph having a perfect $m$-coloring with color adjacency matrix $A$.

On the other hand, if $A$ has multiple blocks, we perform the procedure above
for each block of $A$, and let $G$ be the union of the graphs formed 
for each block. Then $G$ has a perfect $m$-coloring with color adjacency matrix $A$.
Moreover, $G$ is necessarily disconnected, as each vertex in 
$G$ corresponding to one block of $A$ is not adjacent to any vertex of 
$G$ corresponding to any other block of $A$.


This completes the characterization of color adjacency matrices.
\end{proof}

\begin{remark} \label{rem:conn}
Given a realizable color adjacency matrix $A=(a_{ij})_{m \times m}$, 
let $G'$ be the directed multigraph having the color classes $V_{i}$ 
as vertices and $A$ as adjacency matrix. We note that $A$ is 
permutation-conjugate to a block diagonal matrix with more than one block 
if and only if $G'$ is disconnected. 
In our ensuing computations to find candidate 
color adjacency matrices $A$ for connected graphs $G$, 
we use the equivalent fact that $G'$ must also have a spanning tree. Thus,
there is an arrangement $V_{n_{1}}$, $V_{n_{2}}$, $\ldots$, $V_{n_{m}}$
of the color classes such that for $i \geq 1$, $V_{n_{i+1}}$ is connected to 
$V_{n_{j}}$ for some $j \leq i$. In other words, there exists an $(m-1)$-tuple  
$(a_{n_{1},n_{2}},a_{n'_{3},n_{3}},\ldots,a_{n'_{m},n_{m}})$
of nonzero entries such that $n_{1} \neq n_{2}$, and if $m \geq 3$,
$n'_{i} \in \{n_{1},n_{2},\ldots,n_{i-1}\}$ and 
$n_{i} \notin \{n_{1},n_{2},\ldots,n_{i-1}\}$
for $3 \leq i \leq m$.
\end{remark}

\section{Counting lemmas}\label{sec:vercount}
For $m \in \{2,3,4\}$, the next lemmas count the number of vertices in each color class. 
Recall that there may be several ways to express $v_i$ in terms of $v_j$ and the 
$a_{n_r,n_s}$'s depending on the different ways how the induced graph $G'$
(see above) might be connected. Let $G''$ be the simple graph induced by
the multigraph $G'$ above by identifying multiple edges and removing loops.
By Cayley's tree formula the number of possible spanning 
trees of $G''$ equals $m^{m-2}$ \cite{C}. Hence there is only $2^0 = 1$ case to 
consider for $m=2$, there are $3^1$ cases for $m=3$, and $4^2$ cases for $m=4$.
The case $m=2$ (Lemma \ref{lem:vercount2}) appears in \cite{AK}. 
Hence we sketch a proof only for the case when $m=3$. 

\begin{lemma}\label{lem:vercount2}
Let $A= (a_{ij}) \in \N^{2 \times 2}$ be a color adjacency matrix of 
some connected graph $G=(V,E)$. Then $a_{12}$ and $a_{21}$
are both nonzero, and if $v_{i}$ denotes the number of vertices of color $i$, then
\[ v_{1} = \frac{|V|}{1+\frac{a_{12}}{a_{21}}}, 
v_{2} = \frac{|V|}{\frac{a_{21}}{a_{12}}+1}.  \]
\end{lemma}

For $m=3$, while the graph $G'$ may have multiple spanning trees, 
there is only one up to isomorphism, namely, a path on three vertices. 
Dealing with the three cases determined by Cayley's formula 
may then be summarized in the following lemma.

\begin{lemma}\label{lem:vercount3}
Let $A= (a_{ij}) \in \N^{3 \times 3}$ be a color adjacency matrix of some
connected graph $G=(V,E)$. Then there is a permutation $(n_{1} \, n_{2} \, n_{3})$ of 
$(1 \, 2 \, 3)$ such that $a_{n_{1},n_{2}}a_{n_{1},n_{3}} \neq 0$. 
If $v_i$ denotes the number of vertices of color $i$, then
\begin{align*}
v_{n_{1}} & = \frac{|V|}{1+\frac{a_{n_{1},n_{2}}}{a_{n_{2},n_{1}}}+\frac{a_{n_{1},n_{3}}}{a_{n_{3},n_{1}}}}, \\
v_{n_{2}} & = \frac{|V|}{\frac{a_{n_{2},n_{1}}}{a_{n_{1},n_{2}}}+1+\frac{a_{n_{2},n_{1}}a_{n_{1},n_{3}}}{a_{n_{1},n_{2}}a_{n_{3},n_{1}}}}, \\
v_{n_{3}} & = \frac{|V|}{\frac{a_{n_{3},n_{1}}}{a_{n_{1},n_{3}}}+\frac{a_{n_{3},n_{1}}a_{n_{1},n_{2}}}{a_{n_{1},n_{3}}a_{n_{2},n_{1}}}+1}. 
\end{align*}
\end{lemma}

The permutation referred to in Lemma \ref{lem:vercount3} is determined by a spanning tree in $G'$. 
As for the values enumerated, the proof goes along the lines of counting the total number 
of vertices as $|V|=v_{n_{1}}+v_{n_{2}}+v_{n_{3}}$, considering that we have 
$a_{n_{1},i}v_{n_{1}}=a_{i,n_{1}}v_i$, hence $|V|=v_{n_{1}} + 
\frac{a_{n_{1},n_{2}}}{a_{n_{2},n_{1}}}v_{n_{1}} + \frac{a_{n_{1},n_{3}}}{a_{n_{3},n_{1}}}v_{n_{1}}$. 
Here, $a_{n_{2},n_{3}}$ and $a_{n_{3},n_{2}}$ may equal zero. Hence they cannot 
necessarily be used to relate $v_{n_{2}}$ and $v_{n_{3}}$. But, as in the counting procedure 
in the proof of Theorem \ref{thm:charmat}, one obtains 
$a_{n_{2},n_{1}}a_{n_{1},n_{3}}v_{n_{2}}=a_{n_{1},n_{2}}a_{n_{3},n_{1}}v_{n_{3}}$, and 
consequently the expressions for $v_{n_{2}}$ and $v_{n_{3}}$ above. 

Note that the permutation in Lemma~\ref{lem:vercount3} is not necessarily unique.
But the consistency condition of Theorem~\ref{thm:charmat} ensures that the values obtained
are independent of the choice of permutation. We note that Proposition 2.1 in \cite{AA},
which distinguished the mutually exclusive possibilities,
follows from Theorem~\ref{thm:charmat} and Lemma~\ref{lem:vercount3}.

For the four color case, there are two possible spanning trees up to isomorphism, 
namely, a star graph with three leaves, or a path on four vertices. The
sixteen cases from Cayley's tree formula break down into four star graphs as there 
are four choices for the central vertex, and twelve paths arising from the different 
ways of arranging four objects in a row, up to reversal of order.

\begin{lemma}\label{lem:vercount4}
Let $A= (a_{ij}) \in \N^{4 \times 4}$ be a color adjacency matrix of some
connected graph $G=(V,E)$. Then there is a permutation $(n_{1},n_{2},n_{3},n_{4})$ of 
$(1,2,3,4)$ such that $a_{n_{1},n_{2}}a_{n_{1},n_{3}}a_{n_{1},n_{4}} \neq 0$
or $a_{n_{1},n_{2}}a_{n_{2},n_{3}}a_{n_{3},n_{4}} \neq 0$. 
Let $v_i$ denote the number of vertices of color $i$. 
\begin{enumerate}
\item If $a_{n_{1},n_{2}}a_{n_{1},n_{3}}a_{n_{1},n_{4}} \neq 0$, then
\begin{align*} 
v_{n_{1}} & = \frac{|V|}{1+\frac{a_{n_{1},n_{2}}}{a_{n_{2},n_{1}}}
+\frac{a_{n_{1},n_{3}}}{a_{n_{3},n_{1}}}+\frac{a_{n_{1},n_{4}}}{a_{n_{4},n_{1}}}}, \\
v_{n_{2}} & = \frac{|V|}{\frac{a_{n_{2},n_{1}}}{a_{n_{1},n_{2}}}+1
+\frac{a_{n_{2},n_{1}}a_{n_{1},n_{3}}}{a_{n_{1},n_{2}}a_{n_{3},n_{1}}}
+\frac{a_{n_{2},n_{1}}a_{n_{1},n_{4}}}{a_{n_{1},n_{2}}a_{n_{4},n_{1}}}},\\
v_{n_{3}} &= \frac{|V|}{\frac{a_{n_{3},n_{1}}}{a_{n_{1},n_{3}}}
+\frac{a_{n_{3},n_{1}}a_{n_{1},n_{2}}}{a_{n_{1},n_{3}}a_{n_{2},n_{1}}}+1
+\frac{a_{n_{3},n_{1}}a_{n_{1},n_{4}}}{a_{n_{1},n_{3}}a_{n_{4},n_{1}}}},\\
v_{n_{4}} & = \frac{|V|}{\frac{a_{n_{4},n_{1}}}{a_{n_{1},n_{4}}}
+\frac{a_{n_{4},n_{1}}a_{n_{1},n_{2}}}{a_{n_{1},n_{4}}a_{n_{2},n_{1}}}
+\frac{a_{n_{4},n_{1}}a_{n_{1},n_{3}}}{a_{n_{1},n_{4}}a_{n_{3},n_{1}}}+1}.
\end{align*}
\item If $a_{n_{1},n_{2}}a_{n_{2},n_{3}}a_{n_{3},n_{4}}\neq 0$, then
\begin{align*} 
v_{n_{1}} & = \frac{|V|}{1+\frac{a_{n_{1},n_{2}}}{a_{n_{2},n_{1}}}
+\frac{a_{n_{1},n_{2}}a_{n_{2},n_{3}}}{a_{n_{2},n_{1}}a_{n_{3},n_{2}}}
+\frac{a_{n_{1},n_{2}}a_{n_{2},n_{3}}a_{n_{3},n_{4}}}{a_{n_{2},n_{1}}a_{n_{3},n_{2}}a_{n_{4},n_{3}}}}, \\
v_{n_{2}} & = \frac{|V|}{\frac{a_{n_{2},n_{1}}}{a_{n_{1},n_{2}}}
+1+\frac{a_{n_{2},n_{3}}}{a_{n_{3},n_{2}}}
+\frac{a_{n_{2},n_{3}}a_{n_{3},n_{4}}}{a_{n_{3},n_{2}}a_{n_{4},n_{3}}}},\\
v_{n_{3}} &= \frac{|V|}{\frac{a_{n_{3},n_{2}}a_{n_{2},n_{1}}}{a_{n_{2},n_{3}}a_{n_{1},n_{2}}}
+\frac{a_{n_{3},n_{2}}}{a_{n_{2},n_{3}}}+1
+\frac{a_{n_{3},n_{4}}}{a_{n_{4},n_{3}}}},\\
v_{n_{4}} &= \frac{|V|}{\frac{a_{n_{4},n_{3}}a_{n_{3},n_{2}}a_{n_{2},n_{1}}}{a_{n_{3},n_{4}}a_{n_{2},n_{3}}a_{n_{1},n_{2}}}+
\frac{a_{n_{4},n_{3}}a_{n_{3},n_{2}}}{a_{n_{3},n_{4}}a_{n_{2},n_{3}}}+\frac{a_{n_{4},n_{3}}}{a_{n_{3},n_{4}}}+1}.
\end{align*}
\end{enumerate}
\end{lemma}
The proof of this lemma is in complete analogy to the proof of Lemma~\ref{lem:vercount3}:
We count $|V|$ by $|V|=v_{n_1}+v_{n_2}+v_{n_3}+v_{n_4}$. Then we express for instance 
$v_{n_1}$ in terms of $v_{n_2}, v_{n_3}, v_{n_4}$ using the appropriate
(nonzero) $a_{ij}$'s, depending on the possible spanning tree for $G''$. 

\section{Implementation} \label{sec:implement}
Let $G$ be a $k$-regular connected graph and $A \in \N^{m\times m}$ be a 
color adjacency matrix for a perfect $m$-coloring of $G$. Clearly 
each row sum of $A$ equals $k$. Note that for each row there are 
$\binom{k+m-1}{m-1}$ different ways to distribute the entries such 
that the row sum equals $k$. Hence there are $\binom{k+m-1}{m-1}^m$ matrices
to consider altogether. 

The conditions above yield the following procedure to enumerate 
all color adjacency matrices for connected regular graphs. We need $m(m-1)$ nested 
loops to go through all matrices $A=(a_{ij}) \in \N^{m \times m}$ 
with constant row sum $k$.

\begin{enumerate}
\item Check for $i \neq j$ whether it is true that ``$a_{ij}=0$ if and only if $a_{ji}=0$''
(weak symmetry condition of Theorem \ref{thm:charmat}).

\item Ensure connectedness by applying Remark~\ref{rem:conn}.
For $m=2$, this means $a_{12}$ must be nonzero; for $m=3$ and $m=4$, 
we find the right permutation such that the product in the condition of
the corresponding lemma is nonzero. 

\item Check whether the consistency condition of Theorem \ref{thm:charmat} is satisfied 
by going through all relevant products. Several of these products
may be zero, but connectedness implies that there is a way to relate any $v_{r}$ 
and $v_{s}$ by products of nonzero $a_{ij}$'s as in Remark~\ref{rem:eqcon}. \medskip

\noindent Performing the preceding steps for given $m$ and $k$ yields all color adjacency matrices
for perfect $m$-colorings of connected $k$-regular graphs. The following steps are merely 
for removal of matrices that essentially the same partitions, just with the 
colors permuted.  Without loss of generality, we also adopt the convention that 
$v_{i} \leq v_{i+1}$ for $i < m$.

\item We identify a suitable case in Lemma~\ref{lem:vercount2}, 
\ref{lem:vercount3}, or \ref{lem:vercount4} that $A$ satisfies. 
Again, when using Lemma~\ref{lem:vercount3} or \ref{lem:vercount4}, 
by the consistency condition, it is
enough to consider only one spanning tree of $G'$. 
Observe that for each $i$, the value of $\dfrac{v_{i}}{|V|}$ depends
only on the entries $a_{ij}$. We check whether these expressions are in 
nondecreasing order when arranged according to increasing $i$. For the two-color case, 
it suffices to check if $a_{12} \leq a_{21}$. 

\item Finally, we identify matrices if they are conjugate via a permutation matrix.
\end{enumerate}

We summarize our procedure in Algorithm \ref{alg:cam}.

\begin{algorithm}
  \caption{Generate color adjacency matrices of perfect $m$-colorings of connected $k$-regular graphs}\label{alg:cam}
  \begin{algorithmic}[1]
    \Procedure{CAM}{$m,k$}
      \State $L' = $ empty list 
      \For{$A \in \N^{m \times m}$ with constant row sum $k$}
        \If{$a_{ij}=0$ if and only if $a_{ji}=0$ for $i \neq j$}
          \If{there exists $(m-1)$-tuple satisfying Remark~\ref{rem:conn}}
          	\If{for every nontrivial cycle in $S_{m}$, the equation in condition (2) of Theorem~\ref{thm:charmat} is satisfied}
          	  \If{$\frac{v_{i}}{|V|} \leq \frac{v_{i+1}}{|V|}$ for each $i<m$ in Lemma \ref{lem:vercount2}, \ref{lem:vercount3}, or \ref{lem:vercount4}}
          		\State{Add $A$ to $L'$}
          	  \EndIf
          	\EndIf
          \EndIf
        \EndIf
      \EndFor
      \State $L = $ list containing first element of $L'$
      \For{$A \in L'$, starting with second element}
        \If{for each permutation matrix $P$, $PAP^{-1}$ is not equal to any element of $L$}
          \State{Add A to $L$}
        \EndIf
      \EndFor
      \State \textbf{return} $L$
    \EndProcedure
  \end{algorithmic}
\end{algorithm}

The tests above were implemented both in \texttt{scilab} and
\texttt{sagemath} \cite{S}. The \texttt{sagemath} worksheets are available
for download \cite{sws}. There are three worksheets, one for each number
of colors. The worksheets are organized in sections, one for each degree
$k$ of regularity ($k \in \{3,4,5\}$). The comments in the code indicate the 
different cases and tests. After executing all cells in all sections in
the worksheet the list \texttt{l}
contains all color adjacency matrices passing the tests (1.)-(5.) for the respective
value of $k$. Each section contains further code to determine all perfect
colorings of Platonic graphs, see Section \ref{sec:platonic}. 

The worksheets for two and three colors will need at most a few minutes
computing time on an ordinary laptop or desktop computer. The worksheets
for four colors need several hours of computation on a modern laptop.
The most time-consuming part is step 5. Therefore we also provide a sage
data file and a pdf file containing all color adjacency matrices for
download \cite{sws}. One can download the sage data file (for instance
\texttt{4col-list.sage}), store them in some folder (for instance
\texttt{/home/user/sage}) and load the content into any sage worksheet 
using \texttt{open('/home/user/sage/4col-list.sage')}, for instance.
After executing the above command, the list \texttt{l43} contains all
color adjacency matrices for perfect 4-colorings of 3-regular graphs,
\texttt{l44} contains all color adjacency matrices for perfect 4-colorings 
of 4-regular graphs, and \texttt{l45} the corresponding list 
for perfect 4-colorings of 5-regular graphs. These lists can then be
processed further, as seen in the examples in Section~\ref{sec:platonic}. 

\section{Color adjacency matrices of $k$-regular graphs}\label{sec:enum}

Using these criteria all color adjacency matrices $A$ for perfect 
2-colorings of $k$-regular graphs with $k \in \{3,4,5\}$ are only the ones
listed in the Table \ref{tab:2-col}.

\begin{table}[h]
\begin{center}
\begin{tabular}{c|l}
$k$ & \hspace{2in} $A$ \\ \hline
3 & $\big( \begin{smallmatrix} 0 & 3\\ 1 & 2 \end{smallmatrix} \big)$,
 $\big( \begin{smallmatrix} 0 & 3\\ 2 & 1 \end{smallmatrix} \big)$,
$\big( \begin{smallmatrix} 0 & 3\\ 3 & 0 \end{smallmatrix} \big)$,
$\big( \begin{smallmatrix} 1 & 2\\ 1 & 2 \end{smallmatrix} \big)$,
$\big( \begin{smallmatrix} 1 & 2\\ 2 & 1 \end{smallmatrix} \big)$,
$\big( \begin{smallmatrix} 2 & 1\\ 1 & 2 \end{smallmatrix} \big)$\\
  \hline
4 & $\big( \begin{smallmatrix} 0 & 4\\ 1 & 3 \end{smallmatrix} \big)$,
 $\big( \begin{smallmatrix} 0 & 4\\ 2 & 2 \end{smallmatrix} \big)$,
$\big( \begin{smallmatrix} 0 & 4\\ 3 & 1 \end{smallmatrix} \big)$,
$\big( \begin{smallmatrix} 0 & 4\\ 4 & 0 \end{smallmatrix} \big)$,
$\big( \begin{smallmatrix} 1 & 3\\ 1 & 3 \end{smallmatrix} \big)$,
$\big( \begin{smallmatrix} 1 & 3\\ 2 & 2 \end{smallmatrix} \big)$,
$\big( \begin{smallmatrix} 1 & 3\\ 3 & 1 \end{smallmatrix} \big)$,
$\big( \begin{smallmatrix} 2 & 2\\ 1 & 3 \end{smallmatrix} \big)$,
$\big( \begin{smallmatrix} 2 & 2\\ 2 & 2 \end{smallmatrix} \big)$,
$\big( \begin{smallmatrix} 3 & 1\\ 1 & 3 \end{smallmatrix} \big)$\\
  \hline
5 & $\big( \begin{smallmatrix} 0 & 5\\ 1 & 4 \end{smallmatrix} \big)$,
 $\big( \begin{smallmatrix} 0 & 5\\ 2 & 3 \end{smallmatrix} \big)$,
$\big( \begin{smallmatrix} 0 & 5\\ 3 & 2 \end{smallmatrix} \big)$,
$\big( \begin{smallmatrix} 0 & 5\\ 4 & 1 \end{smallmatrix} \big)$,
$\big( \begin{smallmatrix} 0 & 5\\ 5 & 0 \end{smallmatrix} \big)$,
$\big( \begin{smallmatrix} 1 & 4\\ 1 & 4 \end{smallmatrix} \big)$,
 $\big( \begin{smallmatrix} 1 & 4\\ 2 & 3 \end{smallmatrix} \big)$,
$\big( \begin{smallmatrix} 1 & 4\\ 3 & 2 \end{smallmatrix} \big)$,
$\big( \begin{smallmatrix} 1 & 4\\ 4 & 1 \end{smallmatrix} \big)$,
$\big( \begin{smallmatrix} 2 & 3\\ 1 & 4 \end{smallmatrix} \big)$,\\
 & $\big( \begin{smallmatrix} 2 & 3\\ 2 & 3 \end{smallmatrix} \big)$,
 $\big( \begin{smallmatrix} 2 & 3\\ 3 & 2 \end{smallmatrix} \big)$,
$\big( \begin{smallmatrix} 3 & 2\\ 1 & 4 \end{smallmatrix} \big)$,
$\big( \begin{smallmatrix} 3 & 2\\ 2 & 3 \end{smallmatrix} \big)$,
$\big( \begin{smallmatrix} 4 & 1\\ 1 & 4 \end{smallmatrix} \big)$
\end{tabular}
\end{center}
\caption{All color adjacency matrices $A$ for $k$-regular graphs with two colors.
\label{tab:2-col}}
\end{table}

All color adjacency matrices $A$ for perfect 3-colorings 
of $k$-regular graphs with $k \in \{3,4,5\}$ are given in Appendix A.
There are 18 possible matrices for 3-regular graphs, 64 for 4-regular 
graphs, and 153 for 5-regular graphs. 

The lists of all color adjacency matrices $A$ for perfect 4-colorings 
of $k$-regular graphs with $k \in \{3,4,5\}$ are quite long: there are
72 matrices for 3-regular graphs, 485 for 4-regular graphs, and
2042 for 5-regular graphs.  They are available online at \cite{sws} in
two forms: as a list in \texttt{pdf}, and as a loadable \texttt{sage} data file,
see Section \ref{sec:implement}.
Table \ref{tab:compare} below compares the number of all matrices in $\N^{m \times m}$
with all row sums equal to $k$ with the number of all color adjacency matrices
for perfect colorings for 4-colorings of $k$-regular graphs with $k \in \{3,4,5\}$. 

\begin{table}[h]
\begin{center}
\begin{tabular}{c|ccc}
$m$ \textbackslash \, $k$ & 3 & 4 & 5\\
\hline
2 & 6 of 16 & 10 of 25 & 15 of 36\\
3 & 18 of 1000 & 64 of 3375 & 153 of 9261\\
4 & 72 of 16 000 & 485 of 1 500 625 & 2042 of 9 834 496
\end{tabular}
\end{center}
\caption{A comparison of the number of all color adjacency matrices
  for perfect colorings of connected graphs (passing the tests (1.)-(5.))
  with the number of matrices in 
$\N^{m \times m}$ with all row sums equal to $k$. \label{tab:compare}}
\end{table}

\section{Perfect colorings of Platonic graphs} \label{sec:platonic}

Theorem~\ref{thm:eigenv} may now be used as a further necessary criterion for
possible color adjacency matrices for a particular graph $G$. We illustrate this 
with the Platonic graphs (i.e., the edge graphs of the Platonic solids).
The eigenvalues of these graphs  are given in Table \ref{tab:eigen}. These values
can be found for instance in \cite{CDS}. An entry $a^n$ means that
$a$ is an eigenvalue of algebraic multiplicity $n$.

\begin{table}[h]
\[ \begin{array}{l|l}
\hspace{1cm} G & \hspace{1.5cm} \mbox{eigenvalues} \\
\hline
\mbox{tetrahedron} & -1^3,3 \\ 
\mbox{cube} & -3,-1^3,1^3,3 \\
\mbox{octahedron} & -2^2,0^3,4 \\ 
\mbox{dodecahedron} & -\sqrt{5}^{\, 3}, -2^4, 0^4, 1^5, \sqrt{5}^{\, 3}, 3\\
\mbox{icosahedron}  & -\sqrt{5}^{\, 3}, -1^5, \sqrt{5}^{\, 3}, 5
\end{array}  \]
\caption{The eigenvalues of the Platonic graphs. A superscript denotes
the multiplicity of the respective eigenvalue. \label{tab:eigen}}
\end{table}

It follows from Theorem~\ref{thm:eigenv} that to determine all perfect 
colorings of the Platonic graphs with two colors one can check which of the
matrices in Table~\ref{tab:2-col} have eigenvalues in the respective spectrum
of the Platonic graphs. This is the actual test we 
implemented in \texttt{sagemath}. One could refine it in order to include 
counting the multiplicities, but we found by inspection that for these graphs
the latter condition does not exclude further matrices.

\subsection{The perfect 2-colorings of Platonic graphs} \label{sec:2colPla}
By the methods described above, we obtained a list of all color adjacency 
matrices for perfect 2-colorings of $k$-regular graphs for $k \in \{3,4,5\}$,
see Table \ref{tab:2-col}.
For each matrix in each of these lists we now check whether the corresponding 
expressions for $v_i$ in Lemma \ref{lem:vercount2} are integers, and whether the eigenvalues
of the matrix are eigenvalues of the Platonic graph under consideration.
For example, since the cube graph is 3-regular, we check for all six matrices
in the first row of Table \ref{tab:2-col} whether the expressions in Lemma 
\ref{lem:vercount2} are all integers, and whether all eigenvalues of the matrix
are contained in $\{-3, -1, 1, 3\}$. In this manner 
we obtained the following candidates for color adjacency matrices for 2-colorings
of the Platonic graphs, respectively.
\begin{enumerate}
\item Tetrahedron:
$\big( \begin{smallmatrix} 0 & 3\\ 1 & 2 \end{smallmatrix} \big),
\big( \begin{smallmatrix} 1 & 2\\ 2 & 1 \end{smallmatrix} \big)$
\item Cube:
$\big( \begin{smallmatrix} 0 & 3\\ 1 & 2 \end{smallmatrix} \big),
\big( \begin{smallmatrix} 0 & 3\\ 3 & 0 \end{smallmatrix} \big),
\big( \begin{smallmatrix} 1 & 2\\ 2 & 1 \end{smallmatrix} \big),
\big( \begin{smallmatrix} 2 & 1\\ 1 & 2 \end{smallmatrix} \big)$
\item Octahedron:
$\big( \begin{smallmatrix} 0 & 4\\ 2 & 2 \end{smallmatrix} \big),
\big( \begin{smallmatrix} 1 & 3\\ 3 & 1 \end{smallmatrix} \big)^{\dagger},
\big( \begin{smallmatrix} 2 & 2\\ 2 & 2 \end{smallmatrix} \big)$
\item Dodecahedron: 
$\big( \begin{smallmatrix} 0 & 3\\ 2 & 1 \end{smallmatrix} \big),
\big( \begin{smallmatrix} 2 & 1\\ 1 & 2 \end{smallmatrix} \big)$
\item Icosahedron:
$\big( \begin{smallmatrix} 0 & 5\\ 1 & 4 \end{smallmatrix} \big),
\big( \begin{smallmatrix} 1 & 4\\ 2 & 3 \end{smallmatrix} \big),
\big( \begin{smallmatrix} 2 & 3\\ 3 & 2 \end{smallmatrix} \big)$
\end{enumerate}
For the tetrahedron, the cube, the dodecahedron, and the icosahedron,
all possible color adjacency matrices in the list above actually correspond 
to perfect 2-colorings. These colorings are shown in Figures
\ref{fig:tetrahedron}, \ref{fig:cube}, \ref{fig:dodecahedron} and 
\ref{fig:icosahedron}.
For the octahedron there are only two perfect 2-colorings, shown
in Figure \ref{fig:octahedron}. In this case, one of the
matrices above does not correspond to a perfect 2-coloring of the octahedral graph:
the matrix marked
with $\dagger$ can be checked to be impossible in a straightforward manner
by attempting to color the vertices of an octahedral graph according to 
these color adjacencies. One may also argue combinatorially: if this matrix 
is a color adjacency matrix for a connected graph $G$, 
then $G$ must have at least $8$ vertices.
This is because $a_{12}=a_{21}=3$ imply $v_{1}$, $v_{2} \geq 3$ and
$a_{11}=a_{22}=1$ imply $v_{1}$ and $v_{2}$ must be even. Thus,
$G$ cannot be the octahedral graph.
In any case, the list confirms the results in \cite{AK}.

\subsection{The perfect 3-colorings of Platonic graphs}
Applying the analogous procedure, and with Lemma \ref{lem:vercount3} rather than
Lemma \ref{lem:vercount2}, we obtain a list of all color adjacency 
matrices for 3-colorings of the Platonic graphs, respectively. In this case,
all candidates are valid color adjacency matrices for perfect colorings of
Platonic graphs.

\begin{enumerate}
\item Tetrahedron:
$\left( \begin{smallmatrix} 0 & 1 & 2\\ 1 & 0 & 2 \\ 1 & 1 & 1 \end{smallmatrix} \right)$
\item Cube:
$\left( \begin{smallmatrix} 0 & 1 & 2\\ 1 & 0 & 2 \\ 1 & 1 & 1 \end{smallmatrix} \right),
\left( \begin{smallmatrix} 1 & 0 & 2\\ 0 & 1 & 2\\ 1 & 1 & 1 \end{smallmatrix} \right)$
\item Octahedron:
$\left( \begin{smallmatrix} 0 & 0 & 4\\ 0 & 0 & 4\\ 1 & 1 & 2 \end{smallmatrix} \right),
\left( \begin{smallmatrix} 0 & 2 & 2\\ 2 & 0 & 2 \\ 2 & 2 & 0 \end{smallmatrix} \right),
\left( \begin{smallmatrix} 0 & 2 & 2\\ 2 & 1 & 1 \\ 2 & 1 & 1\end{smallmatrix} \right)$
\item Dodecahedron: 
$\left( \begin{smallmatrix} 0 & 0 & 3\\ 0 & 0 & 3\\ 1 & 1 & 1 \end{smallmatrix} \right),
\left( \begin{smallmatrix} 0 & 3 & 0\\ 1 & 0 & 2\\ 0 & 1 & 2 \end{smallmatrix} \right),
\left( \begin{smallmatrix} 1 & 0 & 2\\ 0 & 1 & 2\\ 1 & 2 & 0 \end{smallmatrix} \right)$
\item Icosahedron:
$\left( \begin{smallmatrix} 0 & 1 & 4\\ 1 & 0 & 4\\ 1 & 1 & 3 \end{smallmatrix} \right),
\left( \begin{smallmatrix} 0 & 2 & 3\\ 1 & 1 & 3\\ 1 & 2 & 2 \end{smallmatrix} \right),
\left( \begin{smallmatrix} 1 & 2 & 2\\ 2 & 1 & 2\\ 2 & 2 & 1 \end{smallmatrix} \right)$
\end{enumerate}
The perfect colorings corresponding to the color adjacency matrices above are
shown in Figures \ref{fig:tetrahedron}-\ref{fig:icosahedron}. This list
corrects a preprint version of \cite{AA} by providing the three cases missing there, namely
$\big( \begin{smallmatrix} 0 & 2 & 2\\ 2 & 0 & 2 \\ 2 & 2 & 0 \end{smallmatrix} \big)$ 
for the octahedral graph and 
$\big( \begin{smallmatrix} 0 & 3 & 0\\ 1 & 0 & 2\\ 0 & 1 & 2 \end{smallmatrix} \big)$
and $\big( \begin{smallmatrix} 1 & 0 & 2\\ 0 & 1 & 2\\ 1 & 2 & 0 \end{smallmatrix} \big)$
for the dodecahedral graph. The final version of \cite{AA} is correct.

\subsection{The perfect 4-colorings of Platonic graphs}

We obtained in a similar manner the following candidates for color adjacency matrices 
for 4-colorings of the Platonic graphs, respectively.
\begin{enumerate}
\item Tetrahedron:
$\left(\begin{smallmatrix}0&1&1&1\\ 1&0&1&1\\ 1&1&0&1\\ 1&1&1&0\\ \end{smallmatrix}\right)$
\item Cube:
$\left(\begin{smallmatrix}0&0&0&3\\ 0&0&3&0\\ 0&1&0&2\\ 1&0&2&0\\ \end{smallmatrix}\right)$,
$\left(\begin{smallmatrix}0&0&1&2\\ 0&0&2&1\\ 1&2&0&0\\ 2&1&0&0\\ \end{smallmatrix}\right)$,
$\left(\begin{smallmatrix}0&1&1&1\\ 1&0&1&1\\ 1&1&0&1\\ 1&1&1&0\\ \end{smallmatrix}\right)$,
$\left(\begin{smallmatrix}0&1&1&1\\ 1&0&1&1\\ 1&1&1&0\\ 1&1&0&1\\ \end{smallmatrix}\right)$,
$\left(\begin{smallmatrix}1&0&1&1\\ 0&1&1&1\\ 1&1&1&0\\ 1&1&0&1\\ \end{smallmatrix}\right)$
\item Octahedron:
$\left(\begin{smallmatrix}0&0&2&2\\ 0&0&2&2\\ 1&1&0&2\\ 1&1&2&0\\ \end{smallmatrix}\right)$,
$\left(\begin{smallmatrix}0&0&2&2\\ 0&0&2&2\\ 1&1&1&1\\ 1&1&1&1\\ \end{smallmatrix}\right)$
\item Dodecahedron: 
$\left(\begin{smallmatrix}0&0&0&3\\ 0&0&2&1\\ 0&2&0&1\\ 1&1&1&0\\ \end{smallmatrix}\right)$,
$\left(\begin{smallmatrix}0&0&0&3\\ 0&1&1&1\\ 0&1&1&1\\ 1&1&1&0\\ \end{smallmatrix}\right)$,
$\left(\begin{smallmatrix}0&0&1&2\\ 0&0&1&2\\ 1&1&1&0\\ 1&1&0&1\\ \end{smallmatrix}\right)$,
$\left(\begin{smallmatrix}0&0&1&2\\ 0&2&0&1\\ 1&0&2&0\\ 2&1&0&0\\ \end{smallmatrix}\right)$,
$\left(\begin{smallmatrix}1&0&0&2\\ 0&1&0&2\\ 0&0&1&2\\ 1&1&1&0\\ \end{smallmatrix}\right)$
\item Icosahedron:
$\left(\begin{smallmatrix}0&0&0&5\\ 0&0&5&0\\ 0&1&2&2\\ 1&0&2&2\\ \end{smallmatrix}\right)$,
$\left(\begin{smallmatrix}0&1&1&3\\ 1&0&1&3\\ 1&1&0&3\\ 1&1&1&2\\ \end{smallmatrix}\right)$,
$\left(\begin{smallmatrix}0&1&2&2\\ 1&0&2&2\\ 1&1&1&2\\ 1&1&2&1\\ \end{smallmatrix}\right)$,
$\left(\begin{smallmatrix}0&1&2&2\\ 1&2&0&2\\ 2&0&2&1\\ 2&2&1&0\\ \end{smallmatrix}\right)$
\end{enumerate}
All of these candidates have corresponding perfect colorings, and these
are shown in Figures \ref{fig:tetrahedron}-\ref{fig:icosahedron}, respectively.

\section{Further questions}

The results and methods above give rise to several questions.

\begin{enumerate}
\item Using the lists of realizable color adjacency matrices generated 
in Section~\ref{sec:enum}, one may also try to determine
the perfect colorings of other regular graphs starting with
special classes of graphs, say the Archimedean graphs. These graphs are all
regular with valency at most $5$. 
\item The matrix marked $\dagger$ in Section~\ref{sec:2colPla} could have been
excluded from the list by adding conditions checking if the order of the
color class $V_{i}$ is even if the corresponding diagonal entry is odd. 
Then, the procedure becomes sufficient to enumerate
the realizable color adjacency matrices for Platonic graphs. It would be interesting
to understand why this is so, and to characterize all regular graphs for which this 
modified method is sufficient. 
\item Recall that not all perfect colorings 
correspond to orbit partitions. We then ask if
there are conditions under which a given realizable color adjacency matrix
corresponds to an orbit partition of a graph. For this question it might be
instructive to start with graphs possessing high degrees of symmetry, vertex
transitivity, and edge transitivity.
\end{enumerate}


\bibliographystyle{amsplain}


\section*{Appendix A: All color adjacency matrices for 3-colorings}
All color adjacency matrices $A$ for perfect 3-colorings 
of $k$-regular graphs with $k \in \{3,4,5\}$:

\parindent0em

\medskip

\underline{3-regular graphs:}

$\left(\begin{smallmatrix}0&0&3\\ 0&0&3\\ 1&1&1\\ \end{smallmatrix}\right)$
$\left(\begin{smallmatrix}0&0&3\\ 0&0&3\\ 1&2&0\\ \end{smallmatrix}\right)$
$\left(\begin{smallmatrix}0&0&3\\ 0&1&2\\ 1&1&1\\ \end{smallmatrix}\right)$
$\left(\begin{smallmatrix}0&0&3\\ 0&1&2\\ 1&2&0\\ \end{smallmatrix}\right)$
$\left(\begin{smallmatrix}0&0&3\\ 0&2&1\\ 1&1&1\\ \end{smallmatrix}\right)$
$\left(\begin{smallmatrix}0&0&3\\ 0&2&1\\ 2&1&0\\ \end{smallmatrix}\right)$
$\left(\begin{smallmatrix}0&1&2\\ 1&0&2\\ 1&1&1\\ \end{smallmatrix}\right)$
$\left(\begin{smallmatrix}0&1&2\\ 1&1&1\\ 2&1&0\\ \end{smallmatrix}\right)$
$\left(\begin{smallmatrix}0&1&2\\ 1&2&0\\ 1&0&2\\ \end{smallmatrix}\right)$
$\left(\begin{smallmatrix}0&1&2\\ 1&2&0\\ 2&0&1\\ \end{smallmatrix}\right)$
$\left(\begin{smallmatrix}0&3&0\\ 1&0&2\\ 0&1&2\\ \end{smallmatrix}\right)$
$\left(\begin{smallmatrix}1&0&2\\ 0&0&3\\ 1&2&0\\ \end{smallmatrix}\right)$
$\left(\begin{smallmatrix}1&0&2\\ 0&1&2\\ 1&1&1\\ \end{smallmatrix}\right)$
$\left(\begin{smallmatrix}1&0&2\\ 0&1&2\\ 1&2&0\\ \end{smallmatrix}\right)$
$\left(\begin{smallmatrix}1&0&2\\ 0&2&1\\ 1&1&1\\ \end{smallmatrix}\right)$
$\left(\begin{smallmatrix}1&1&1\\ 1&1&1\\ 1&1&1\\ \end{smallmatrix}\right)$
$\left(\begin{smallmatrix}1&1&1\\ 1&2&0\\ 1&0&2\\ \end{smallmatrix}\right)$
$\left(\begin{smallmatrix}1&2&0\\ 1&0&2\\ 0&1&2\\ \end{smallmatrix}\right)$


\medskip

\underline{4-regular graphs:}

$\left(\begin{smallmatrix}0&0&4\\ 0&0&4\\ 1&1&2\\ \end{smallmatrix}\right)$
$\left(\begin{smallmatrix}0&0&4\\ 0&0&4\\ 1&2&1\\ \end{smallmatrix}\right)$
$\left(\begin{smallmatrix}0&0&4\\ 0&0&4\\ 1&3&0\\ \end{smallmatrix}\right)$
$\left(\begin{smallmatrix}0&0&4\\ 0&0&4\\ 2&2&0\\ \end{smallmatrix}\right)$
$\left(\begin{smallmatrix}0&0&4\\ 0&1&3\\ 1&1&2\\ \end{smallmatrix}\right)$
$\left(\begin{smallmatrix}0&0&4\\ 0&1&3\\ 1&2&1\\ \end{smallmatrix}\right)$
$\left(\begin{smallmatrix}0&0&4\\ 0&1&3\\ 1&3&0\\ \end{smallmatrix}\right)$
$\left(\begin{smallmatrix}0&0&4\\ 0&1&3\\ 2&2&0\\ \end{smallmatrix}\right)$
$\left(\begin{smallmatrix}0&0&4\\ 0&2&2\\ 1&1&2\\ \end{smallmatrix}\right)$
$\left(\begin{smallmatrix}0&0&4\\ 0&2&2\\ 1&2&1\\ \end{smallmatrix}\right)$
$\left(\begin{smallmatrix}0&0&4\\ 0&2&2\\ 2&1&1\\ \end{smallmatrix}\right)$
$\left(\begin{smallmatrix}0&0&4\\ 0&2&2\\ 2&2&0\\ \end{smallmatrix}\right)$
$\left(\begin{smallmatrix}0&0&4\\ 0&3&1\\ 1&1&2\\ \end{smallmatrix}\right)$
$\left(\begin{smallmatrix}0&0&4\\ 0&3&1\\ 2&1&1\\ \end{smallmatrix}\right)$
$\left(\begin{smallmatrix}0&0&4\\ 0&3&1\\ 3&1&0\\ \end{smallmatrix}\right)$
$\left(\begin{smallmatrix}0&1&3\\ 1&0&3\\ 1&1&2\\ \end{smallmatrix}\right)$
$\left(\begin{smallmatrix}0&1&3\\ 1&0&3\\ 2&2&0\\ \end{smallmatrix}\right)$
$\left(\begin{smallmatrix}0&1&3\\ 1&2&1\\ 3&1&0\\ \end{smallmatrix}\right)$
$\left(\begin{smallmatrix}0&1&3\\ 1&3&0\\ 1&0&3\\ \end{smallmatrix}\right)$
$\left(\begin{smallmatrix}0&1&3\\ 1&3&0\\ 2&0&2\\ \end{smallmatrix}\right)$
$\left(\begin{smallmatrix}0&1&3\\ 1&3&0\\ 3&0&1\\ \end{smallmatrix}\right)$
$\left(\begin{smallmatrix}0&2&2\\ 1&0&3\\ 1&3&0\\ \end{smallmatrix}\right)$
$\left(\begin{smallmatrix}0&2&2\\ 1&1&2\\ 1&2&1\\ \end{smallmatrix}\right)$
$\left(\begin{smallmatrix}0&2&2\\ 1&2&1\\ 1&1&2\\ \end{smallmatrix}\right)$
$\left(\begin{smallmatrix}0&2&2\\ 1&3&0\\ 1&0&3\\ \end{smallmatrix}\right)$
$\left(\begin{smallmatrix}0&2&2\\ 2&0&2\\ 1&1&2\\ \end{smallmatrix}\right)$
$\left(\begin{smallmatrix}0&2&2\\ 2&0&2\\ 2&2&0\\ \end{smallmatrix}\right)$
$\left(\begin{smallmatrix}0&2&2\\ 2&1&1\\ 2&1&1\\ \end{smallmatrix}\right)$
$\left(\begin{smallmatrix}0&2&2\\ 2&2&0\\ 1&0&3\\ \end{smallmatrix}\right)$
$\left(\begin{smallmatrix}0&2&2\\ 2&2&0\\ 2&0&2\\ \end{smallmatrix}\right)$
$\left(\begin{smallmatrix}0&4&0\\ 1&0&3\\ 0&1&3\\ \end{smallmatrix}\right)$
$\left(\begin{smallmatrix}0&4&0\\ 1&0&3\\ 0&2&2\\ \end{smallmatrix}\right)$
$\left(\begin{smallmatrix}0&4&0\\ 1&1&2\\ 0&1&3\\ \end{smallmatrix}\right)$
$\left(\begin{smallmatrix}0&4&0\\ 2&0&2\\ 0&1&3\\ \end{smallmatrix}\right)$
$\left(\begin{smallmatrix}1&0&3\\ 0&0&4\\ 1&2&1\\ \end{smallmatrix}\right)$
$\left(\begin{smallmatrix}1&0&3\\ 0&0&4\\ 1&3&0\\ \end{smallmatrix}\right)$
$\left(\begin{smallmatrix}1&0&3\\ 0&1&3\\ 1&1&2\\ \end{smallmatrix}\right)$
$\left(\begin{smallmatrix}1&0&3\\ 0&1&3\\ 1&2&1\\ \end{smallmatrix}\right)$
$\left(\begin{smallmatrix}1&0&3\\ 0&1&3\\ 1&3&0\\ \end{smallmatrix}\right)$
$\left(\begin{smallmatrix}1&0&3\\ 0&1&3\\ 2&2&0\\ \end{smallmatrix}\right)$
$\left(\begin{smallmatrix}1&0&3\\ 0&2&2\\ 1&1&2\\ \end{smallmatrix}\right)$
$\left(\begin{smallmatrix}1&0&3\\ 0&2&2\\ 1&2&1\\ \end{smallmatrix}\right)$
$\left(\begin{smallmatrix}1&0&3\\ 0&2&2\\ 2&2&0\\ \end{smallmatrix}\right)$
$\left(\begin{smallmatrix}1&0&3\\ 0&3&1\\ 1&1&2\\ \end{smallmatrix}\right)$
$\left(\begin{smallmatrix}1&0&3\\ 0&3&1\\ 2&1&1\\ \end{smallmatrix}\right)$
$\left(\begin{smallmatrix}1&1&2\\ 1&1&2\\ 1&1&2\\ \end{smallmatrix}\right)$
$\left(\begin{smallmatrix}1&1&2\\ 1&2&1\\ 2&1&1\\ \end{smallmatrix}\right)$
$\left(\begin{smallmatrix}1&1&2\\ 1&3&0\\ 1&0&3\\ \end{smallmatrix}\right)$
$\left(\begin{smallmatrix}1&1&2\\ 1&3&0\\ 2&0&2\\ \end{smallmatrix}\right)$
$\left(\begin{smallmatrix}1&3&0\\ 1&0&3\\ 0&1&3\\ \end{smallmatrix}\right)$
$\left(\begin{smallmatrix}1&3&0\\ 1&0&3\\ 0&2&2\\ \end{smallmatrix}\right)$
$\left(\begin{smallmatrix}1&3&0\\ 1&1&2\\ 0&1&3\\ \end{smallmatrix}\right)$
$\left(\begin{smallmatrix}1&3&0\\ 2&0&2\\ 0&1&3\\ \end{smallmatrix}\right)$
$\left(\begin{smallmatrix}2&0&2\\ 0&0&4\\ 1&3&0\\ \end{smallmatrix}\right)$
$\left(\begin{smallmatrix}2&0&2\\ 0&1&3\\ 1&2&1\\ \end{smallmatrix}\right)$
$\left(\begin{smallmatrix}2&0&2\\ 0&1&3\\ 1&3&0\\ \end{smallmatrix}\right)$
$\left(\begin{smallmatrix}2&0&2\\ 0&2&2\\ 1&1&2\\ \end{smallmatrix}\right)$
$\left(\begin{smallmatrix}2&0&2\\ 0&2&2\\ 1&2&1\\ \end{smallmatrix}\right)$
$\left(\begin{smallmatrix}2&0&2\\ 0&3&1\\ 1&1&2\\ \end{smallmatrix}\right)$
$\left(\begin{smallmatrix}2&1&1\\ 1&2&1\\ 1&1&2\\ \end{smallmatrix}\right)$
$\left(\begin{smallmatrix}2&1&1\\ 1&3&0\\ 1&0&3\\ \end{smallmatrix}\right)$
$\left(\begin{smallmatrix}2&2&0\\ 1&0&3\\ 0&1&3\\ \end{smallmatrix}\right)$
$\left(\begin{smallmatrix}2&2&0\\ 1&0&3\\ 0&2&2\\ \end{smallmatrix}\right)$
$\left(\begin{smallmatrix}2&2&0\\ 1&1&2\\ 0&1&3\\ \end{smallmatrix}\right)$

\medskip

\underline{5-regular graphs}

$\left(\begin{smallmatrix}0&0&5\\ 0&0&5\\ 1&1&3\\ \end{smallmatrix}\right)$
$\left(\begin{smallmatrix}0&0&5\\ 0&0&5\\ 1&2&2\\ \end{smallmatrix}\right)$
$\left(\begin{smallmatrix}0&0&5\\ 0&0&5\\ 1&3&1\\ \end{smallmatrix}\right)$
$\left(\begin{smallmatrix}0&0&5\\ 0&0&5\\ 1&4&0\\ \end{smallmatrix}\right)$
$\left(\begin{smallmatrix}0&0&5\\ 0&0&5\\ 2&2&1\\ \end{smallmatrix}\right)$
$\left(\begin{smallmatrix}0&0&5\\ 0&0&5\\ 2&3&0\\ \end{smallmatrix}\right)$
$\left(\begin{smallmatrix}0&0&5\\ 0&1&4\\ 1&1&3\\ \end{smallmatrix}\right)$
$\left(\begin{smallmatrix}0&0&5\\ 0&1&4\\ 1&2&2\\ \end{smallmatrix}\right)$
$\left(\begin{smallmatrix}0&0&5\\ 0&1&4\\ 1&3&1\\ \end{smallmatrix}\right)$
$\left(\begin{smallmatrix}0&0&5\\ 0&1&4\\ 1&4&0\\ \end{smallmatrix}\right)$
$\left(\begin{smallmatrix}0&0&5\\ 0&1&4\\ 2&2&1\\ \end{smallmatrix}\right)$
$\left(\begin{smallmatrix}0&0&5\\ 0&1&4\\ 2&3&0\\ \end{smallmatrix}\right)$
$\left(\begin{smallmatrix}0&0&5\\ 0&2&3\\ 1&1&3\\ \end{smallmatrix}\right)$
$\left(\begin{smallmatrix}0&0&5\\ 0&2&3\\ 1&2&2\\ \end{smallmatrix}\right)$
$\left(\begin{smallmatrix}0&0&5\\ 0&2&3\\ 1&3&1\\ \end{smallmatrix}\right)$
$\left(\begin{smallmatrix}0&0&5\\ 0&2&3\\ 2&2&1\\ \end{smallmatrix}\right)$
$\left(\begin{smallmatrix}0&0&5\\ 0&2&3\\ 2&3&0\\ \end{smallmatrix}\right)$
$\left(\begin{smallmatrix}0&0&5\\ 0&2&3\\ 3&2&0\\ \end{smallmatrix}\right)$
$\left(\begin{smallmatrix}0&0&5\\ 0&3&2\\ 1&1&3\\ \end{smallmatrix}\right)$
$\left(\begin{smallmatrix}0&0&5\\ 0&3&2\\ 1&2&2\\ \end{smallmatrix}\right)$
$\left(\begin{smallmatrix}0&0&5\\ 0&3&2\\ 2&1&2\\ \end{smallmatrix}\right)$
$\left(\begin{smallmatrix}0&0&5\\ 0&3&2\\ 2&2&1\\ \end{smallmatrix}\right)$
$\left(\begin{smallmatrix}0&0&5\\ 0&3&2\\ 3&2&0\\ \end{smallmatrix}\right)$
$\left(\begin{smallmatrix}0&0&5\\ 0&4&1\\ 1&1&3\\ \end{smallmatrix}\right)$
$\left(\begin{smallmatrix}0&0&5\\ 0&4&1\\ 2&1&2\\ \end{smallmatrix}\right)$
$\left(\begin{smallmatrix}0&0&5\\ 0&4&1\\ 3&1&1\\ \end{smallmatrix}\right)$
$\left(\begin{smallmatrix}0&0&5\\ 0&4&1\\ 4&1&0\\ \end{smallmatrix}\right)$
$\left(\begin{smallmatrix}0&1&4\\ 1&0&4\\ 1&1&3\\ \end{smallmatrix}\right)$
$\left(\begin{smallmatrix}0&1&4\\ 1&0&4\\ 2&2&1\\ \end{smallmatrix}\right)$
$\left(\begin{smallmatrix}0&1&4\\ 1&2&2\\ 2&1&2\\ \end{smallmatrix}\right)$
$\left(\begin{smallmatrix}0&1&4\\ 1&3&1\\ 4&1&0\\ \end{smallmatrix}\right)$
$\left(\begin{smallmatrix}0&1&4\\ 1&4&0\\ 1&0&4\\ \end{smallmatrix}\right)$
$\left(\begin{smallmatrix}0&1&4\\ 1&4&0\\ 2&0&3\\ \end{smallmatrix}\right)$
$\left(\begin{smallmatrix}0&1&4\\ 1&4&0\\ 3&0&2\\ \end{smallmatrix}\right)$
$\left(\begin{smallmatrix}0&1&4\\ 1&4&0\\ 4&0&1\\ \end{smallmatrix}\right)$
$\left(\begin{smallmatrix}0&2&3\\ 1&1&3\\ 1&2&2\\ \end{smallmatrix}\right)$
$\left(\begin{smallmatrix}0&2&3\\ 1&4&0\\ 1&0&4\\ \end{smallmatrix}\right)$
$\left(\begin{smallmatrix}0&2&3\\ 2&0&3\\ 1&1&3\\ \end{smallmatrix}\right)$
$\left(\begin{smallmatrix}0&2&3\\ 2&0&3\\ 2&2&1\\ \end{smallmatrix}\right)$
$\left(\begin{smallmatrix}0&2&3\\ 2&1&2\\ 3&2&0\\ \end{smallmatrix}\right)$
$\left(\begin{smallmatrix}0&2&3\\ 2&2&1\\ 3&1&1\\ \end{smallmatrix}\right)$
$\left(\begin{smallmatrix}0&2&3\\ 2&3&0\\ 1&0&4\\ \end{smallmatrix}\right)$
$\left(\begin{smallmatrix}0&2&3\\ 2&3&0\\ 2&0&3\\ \end{smallmatrix}\right)$
$\left(\begin{smallmatrix}0&2&3\\ 2&3&0\\ 3&0&2\\ \end{smallmatrix}\right)$
$\left(\begin{smallmatrix}0&3&2\\ 2&3&0\\ 1&0&4\\ \end{smallmatrix}\right)$
$\left(\begin{smallmatrix}0&3&2\\ 3&0&2\\ 1&1&3\\ \end{smallmatrix}\right)$
$\left(\begin{smallmatrix}0&3&2\\ 3&2&0\\ 1&0&4\\ \end{smallmatrix}\right)$
$\left(\begin{smallmatrix}0&5&0\\ 1&0&4\\ 0&1&4\\ \end{smallmatrix}\right)$
$\left(\begin{smallmatrix}0&5&0\\ 1&0&4\\ 0&2&3\\ \end{smallmatrix}\right)$
$\left(\begin{smallmatrix}0&5&0\\ 1&0&4\\ 0&3&2\\ \end{smallmatrix}\right)$
$\left(\begin{smallmatrix}0&5&0\\ 1&1&3\\ 0&1&4\\ \end{smallmatrix}\right)$
$\left(\begin{smallmatrix}0&5&0\\ 1&1&3\\ 0&2&3\\ \end{smallmatrix}\right)$
$\left(\begin{smallmatrix}0&5&0\\ 1&2&2\\ 0&1&4\\ \end{smallmatrix}\right)$
$\left(\begin{smallmatrix}0&5&0\\ 2&0&3\\ 0&1&4\\ \end{smallmatrix}\right)$
$\left(\begin{smallmatrix}0&5&0\\ 2&0&3\\ 0&2&3\\ \end{smallmatrix}\right)$
$\left(\begin{smallmatrix}0&5&0\\ 2&1&2\\ 0&1&4\\ \end{smallmatrix}\right)$
$\left(\begin{smallmatrix}0&5&0\\ 3&0&2\\ 0&1&4\\ \end{smallmatrix}\right)$
$\left(\begin{smallmatrix}1&0&4\\ 0&0&5\\ 1&2&2\\ \end{smallmatrix}\right)$
$\left(\begin{smallmatrix}1&0&4\\ 0&0&5\\ 1&3&1\\ \end{smallmatrix}\right)$
$\left(\begin{smallmatrix}1&0&4\\ 0&0&5\\ 1&4&0\\ \end{smallmatrix}\right)$
$\left(\begin{smallmatrix}1&0&4\\ 0&0&5\\ 2&3&0\\ \end{smallmatrix}\right)$
$\left(\begin{smallmatrix}1&0&4\\ 0&1&4\\ 1&1&3\\ \end{smallmatrix}\right)$
$\left(\begin{smallmatrix}1&0&4\\ 0&1&4\\ 1&2&2\\ \end{smallmatrix}\right)$
$\left(\begin{smallmatrix}1&0&4\\ 0&1&4\\ 1&3&1\\ \end{smallmatrix}\right)$
$\left(\begin{smallmatrix}1&0&4\\ 0&1&4\\ 1&4&0\\ \end{smallmatrix}\right)$
$\left(\begin{smallmatrix}1&0&4\\ 0&1&4\\ 2&2&1\\ \end{smallmatrix}\right)$
$\left(\begin{smallmatrix}1&0&4\\ 0&1&4\\ 2&3&0\\ \end{smallmatrix}\right)$
$\left(\begin{smallmatrix}1&0&4\\ 0&2&3\\ 1&1&3\\ \end{smallmatrix}\right)$
$\left(\begin{smallmatrix}1&0&4\\ 0&2&3\\ 1&2&2\\ \end{smallmatrix}\right)$
$\left(\begin{smallmatrix}1&0&4\\ 0&2&3\\ 1&3&1\\ \end{smallmatrix}\right)$
$\left(\begin{smallmatrix}1&0&4\\ 0&2&3\\ 2&2&1\\ \end{smallmatrix}\right)$
$\left(\begin{smallmatrix}1&0&4\\ 0&2&3\\ 2&3&0\\ \end{smallmatrix}\right)$
$\left(\begin{smallmatrix}1&0&4\\ 0&3&2\\ 1&1&3\\ \end{smallmatrix}\right)$
$\left(\begin{smallmatrix}1&0&4\\ 0&3&2\\ 1&2&2\\ \end{smallmatrix}\right)$
$\left(\begin{smallmatrix}1&0&4\\ 0&3&2\\ 2&1&2\\ \end{smallmatrix}\right)$
$\left(\begin{smallmatrix}1&0&4\\ 0&3&2\\ 2&2&1\\ \end{smallmatrix}\right)$
$\left(\begin{smallmatrix}1&0&4\\ 0&3&2\\ 3&2&0\\ \end{smallmatrix}\right)$
$\left(\begin{smallmatrix}1&0&4\\ 0&4&1\\ 1&1&3\\ \end{smallmatrix}\right)$
$\left(\begin{smallmatrix}1&0&4\\ 0&4&1\\ 2&1&2\\ \end{smallmatrix}\right)$
$\left(\begin{smallmatrix}1&0&4\\ 0&4&1\\ 3&1&1\\ \end{smallmatrix}\right)$
$\left(\begin{smallmatrix}1&1&3\\ 1&1&3\\ 1&1&3\\ \end{smallmatrix}\right)$
$\left(\begin{smallmatrix}1&1&3\\ 1&1&3\\ 2&2&1\\ \end{smallmatrix}\right)$
$\left(\begin{smallmatrix}1&1&3\\ 1&3&1\\ 3&1&1\\ \end{smallmatrix}\right)$
$\left(\begin{smallmatrix}1&1&3\\ 1&4&0\\ 1&0&4\\ \end{smallmatrix}\right)$
$\left(\begin{smallmatrix}1&1&3\\ 1&4&0\\ 2&0&3\\ \end{smallmatrix}\right)$
$\left(\begin{smallmatrix}1&1&3\\ 1&4&0\\ 3&0&2\\ \end{smallmatrix}\right)$
$\left(\begin{smallmatrix}1&2&2\\ 1&0&4\\ 1&4&0\\ \end{smallmatrix}\right)$
$\left(\begin{smallmatrix}1&2&2\\ 1&1&3\\ 1&3&1\\ \end{smallmatrix}\right)$
$\left(\begin{smallmatrix}1&2&2\\ 1&2&2\\ 1&2&2\\ \end{smallmatrix}\right)$
$\left(\begin{smallmatrix}1&2&2\\ 1&3&1\\ 1&1&3\\ \end{smallmatrix}\right)$
$\left(\begin{smallmatrix}1&2&2\\ 1&4&0\\ 1&0&4\\ \end{smallmatrix}\right)$
$\left(\begin{smallmatrix}1&2&2\\ 2&1&2\\ 1&1&3\\ \end{smallmatrix}\right)$
$\left(\begin{smallmatrix}1&2&2\\ 2&1&2\\ 2&2&1\\ \end{smallmatrix}\right)$
$\left(\begin{smallmatrix}1&2&2\\ 2&2&1\\ 2&1&2\\ \end{smallmatrix}\right)$
$\left(\begin{smallmatrix}1&2&2\\ 2&3&0\\ 1&0&4\\ \end{smallmatrix}\right)$
$\left(\begin{smallmatrix}1&2&2\\ 2&3&0\\ 2&0&3\\ \end{smallmatrix}\right)$
$\left(\begin{smallmatrix}1&4&0\\ 1&0&4\\ 0&1&4\\ \end{smallmatrix}\right)$
$\left(\begin{smallmatrix}1&4&0\\ 1&0&4\\ 0&2&3\\ \end{smallmatrix}\right)$
$\left(\begin{smallmatrix}1&4&0\\ 1&0&4\\ 0&3&2\\ \end{smallmatrix}\right)$
$\left(\begin{smallmatrix}1&4&0\\ 1&1&3\\ 0&1&4\\ \end{smallmatrix}\right)$
$\left(\begin{smallmatrix}1&4&0\\ 1&1&3\\ 0&2&3\\ \end{smallmatrix}\right)$
$\left(\begin{smallmatrix}1&4&0\\ 1&2&2\\ 0&1&4\\ \end{smallmatrix}\right)$
$\left(\begin{smallmatrix}1&4&0\\ 2&0&3\\ 0&1&4\\ \end{smallmatrix}\right)$
$\left(\begin{smallmatrix}1&4&0\\ 2&0&3\\ 0&2&3\\ \end{smallmatrix}\right)$
$\left(\begin{smallmatrix}1&4&0\\ 2&1&2\\ 0&1&4\\ \end{smallmatrix}\right)$
$\left(\begin{smallmatrix}1&4&0\\ 3&0&2\\ 0&1&4\\ \end{smallmatrix}\right)$
$\left(\begin{smallmatrix}2&0&3\\ 0&0&5\\ 1&2&2\\ \end{smallmatrix}\right)$
$\left(\begin{smallmatrix}2&0&3\\ 0&0&5\\ 1&3&1\\ \end{smallmatrix}\right)$
$\left(\begin{smallmatrix}2&0&3\\ 0&0&5\\ 1&4&0\\ \end{smallmatrix}\right)$
$\left(\begin{smallmatrix}2&0&3\\ 0&1&4\\ 1&2&2\\ \end{smallmatrix}\right)$
$\left(\begin{smallmatrix}2&0&3\\ 0&1&4\\ 1&3&1\\ \end{smallmatrix}\right)$
$\left(\begin{smallmatrix}2&0&3\\ 0&1&4\\ 1&4&0\\ \end{smallmatrix}\right)$
$\left(\begin{smallmatrix}2&0&3\\ 0&1&4\\ 2&3&0\\ \end{smallmatrix}\right)$
$\left(\begin{smallmatrix}2&0&3\\ 0&2&3\\ 1&1&3\\ \end{smallmatrix}\right)$
$\left(\begin{smallmatrix}2&0&3\\ 0&2&3\\ 1&2&2\\ \end{smallmatrix}\right)$
$\left(\begin{smallmatrix}2&0&3\\ 0&2&3\\ 1&3&1\\ \end{smallmatrix}\right)$
$\left(\begin{smallmatrix}2&0&3\\ 0&2&3\\ 2&2&1\\ \end{smallmatrix}\right)$
$\left(\begin{smallmatrix}2&0&3\\ 0&2&3\\ 2&3&0\\ \end{smallmatrix}\right)$
$\left(\begin{smallmatrix}2&0&3\\ 0&3&2\\ 1&1&3\\ \end{smallmatrix}\right)$
$\left(\begin{smallmatrix}2&0&3\\ 0&3&2\\ 1&2&2\\ \end{smallmatrix}\right)$
$\left(\begin{smallmatrix}2&0&3\\ 0&3&2\\ 2&2&1\\ \end{smallmatrix}\right)$
$\left(\begin{smallmatrix}2&0&3\\ 0&4&1\\ 1&1&3\\ \end{smallmatrix}\right)$
$\left(\begin{smallmatrix}2&0&3\\ 0&4&1\\ 2&1&2\\ \end{smallmatrix}\right)$
$\left(\begin{smallmatrix}2&1&2\\ 1&2&2\\ 1&1&3\\ \end{smallmatrix}\right)$
$\left(\begin{smallmatrix}2&1&2\\ 1&3&1\\ 2&1&2\\ \end{smallmatrix}\right)$
$\left(\begin{smallmatrix}2&1&2\\ 1&4&0\\ 1&0&4\\ \end{smallmatrix}\right)$
$\left(\begin{smallmatrix}2&1&2\\ 1&4&0\\ 2&0&3\\ \end{smallmatrix}\right)$
$\left(\begin{smallmatrix}2&3&0\\ 1&0&4\\ 0&1&4\\ \end{smallmatrix}\right)$
$\left(\begin{smallmatrix}2&3&0\\ 1&0&4\\ 0&2&3\\ \end{smallmatrix}\right)$
$\left(\begin{smallmatrix}2&3&0\\ 1&0&4\\ 0&3&2\\ \end{smallmatrix}\right)$
$\left(\begin{smallmatrix}2&3&0\\ 1&1&3\\ 0&1&4\\ \end{smallmatrix}\right)$
$\left(\begin{smallmatrix}2&3&0\\ 1&1&3\\ 0&2&3\\ \end{smallmatrix}\right)$
$\left(\begin{smallmatrix}2&3&0\\ 1&2&2\\ 0&1&4\\ \end{smallmatrix}\right)$
$\left(\begin{smallmatrix}2&3&0\\ 2&0&3\\ 0&1&4\\ \end{smallmatrix}\right)$
$\left(\begin{smallmatrix}2&3&0\\ 2&0&3\\ 0&2&3\\ \end{smallmatrix}\right)$
$\left(\begin{smallmatrix}2&3&0\\ 2&1&2\\ 0&1&4\\ \end{smallmatrix}\right)$
$\left(\begin{smallmatrix}3&0&2\\ 0&0&5\\ 1&3&1\\ \end{smallmatrix}\right)$
$\left(\begin{smallmatrix}3&0&2\\ 0&0&5\\ 1&4&0\\ \end{smallmatrix}\right)$
$\left(\begin{smallmatrix}3&0&2\\ 0&1&4\\ 1&3&1\\ \end{smallmatrix}\right)$
$\left(\begin{smallmatrix}3&0&2\\ 0&1&4\\ 1&4&0\\ \end{smallmatrix}\right)$
$\left(\begin{smallmatrix}3&0&2\\ 0&2&3\\ 1&2&2\\ \end{smallmatrix}\right)$
$\left(\begin{smallmatrix}3&0&2\\ 0&2&3\\ 1&3&1\\ \end{smallmatrix}\right)$
$\left(\begin{smallmatrix}3&0&2\\ 0&3&2\\ 1&1&3\\ \end{smallmatrix}\right)$
$\left(\begin{smallmatrix}3&0&2\\ 0&3&2\\ 1&2&2\\ \end{smallmatrix}\right)$
$\left(\begin{smallmatrix}3&0&2\\ 0&4&1\\ 1&1&3\\ \end{smallmatrix}\right)$
$\left(\begin{smallmatrix}3&1&1\\ 1&3&1\\ 1&1&3\\ \end{smallmatrix}\right)$
$\left(\begin{smallmatrix}3&1&1\\ 1&4&0\\ 1&0&4\\ \end{smallmatrix}\right)$
$\left(\begin{smallmatrix}3&2&0\\ 1&0&4\\ 0&1&4\\ \end{smallmatrix}\right)$
$\left(\begin{smallmatrix}3&2&0\\ 1&0&4\\ 0&2&3\\ \end{smallmatrix}\right)$
$\left(\begin{smallmatrix}3&2&0\\ 1&0&4\\ 0&3&2\\ \end{smallmatrix}\right)$
$\left(\begin{smallmatrix}3&2&0\\ 1&1&3\\ 0&1&4\\ \end{smallmatrix}\right)$
$\left(\begin{smallmatrix}3&2&0\\ 1&1&3\\ 0&2&3\\ \end{smallmatrix}\right)$
$\left(\begin{smallmatrix}3&2&0\\ 1&2&2\\ 0&1&4\\ \end{smallmatrix}\right)$

\section*{Appendix B: Perfect colorings of the Platonic graphs}

\begin{figure}[h]
\[ \includegraphics[width=.3\textwidth]{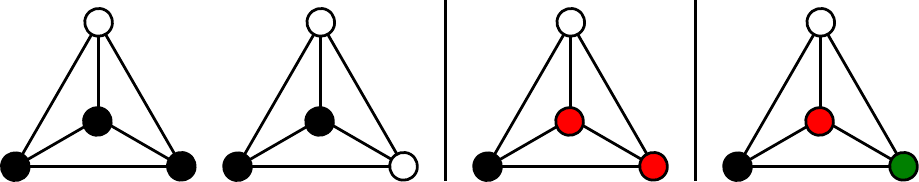} \]
\caption{The perfect 2-, 3- and 4-colorings of the tetrahedral graph.
  Here and in the following figures white is color 1, black is color 2,
  red is color 3 and green is color 4. \label{fig:tetrahedron}}
\end{figure}

\begin{figure}[h]
\[ \includegraphics[width=.5\textwidth]{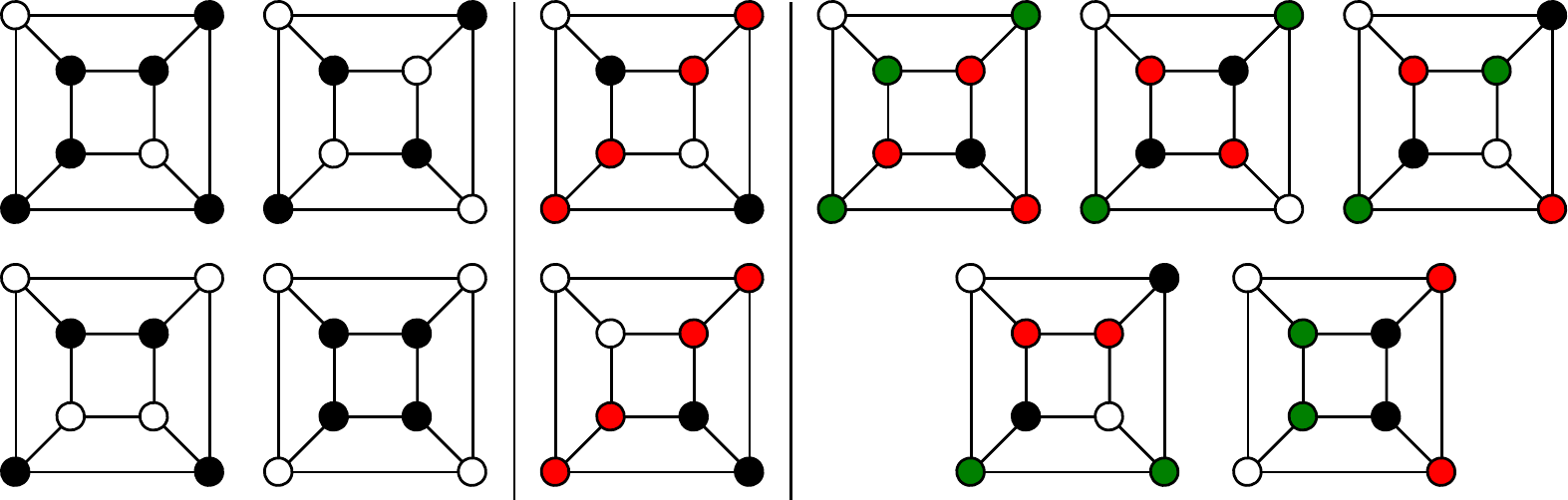} \]
\caption{The perfect 2-, 3- and 4-colorings of the
cube graph. \label{fig:cube}}
\end{figure}

\begin{figure}[h]
\[ \includegraphics[width=.6\textwidth]{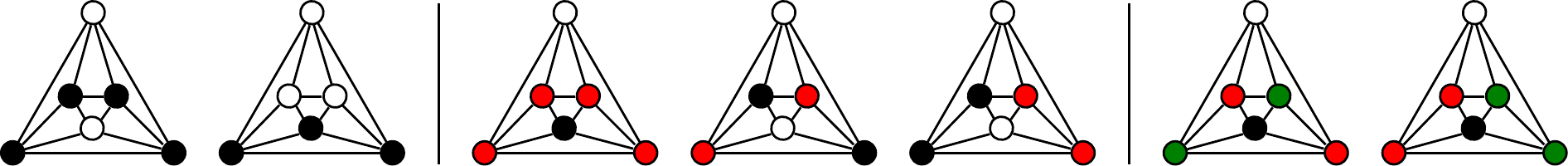} \]
\caption{The perfect 2-, 3- and 4-colorings of the
octahedral graph. \label{fig:octahedron}}
\end{figure}

\begin{figure}[h]
\[ \includegraphics[width=0.6\textwidth]{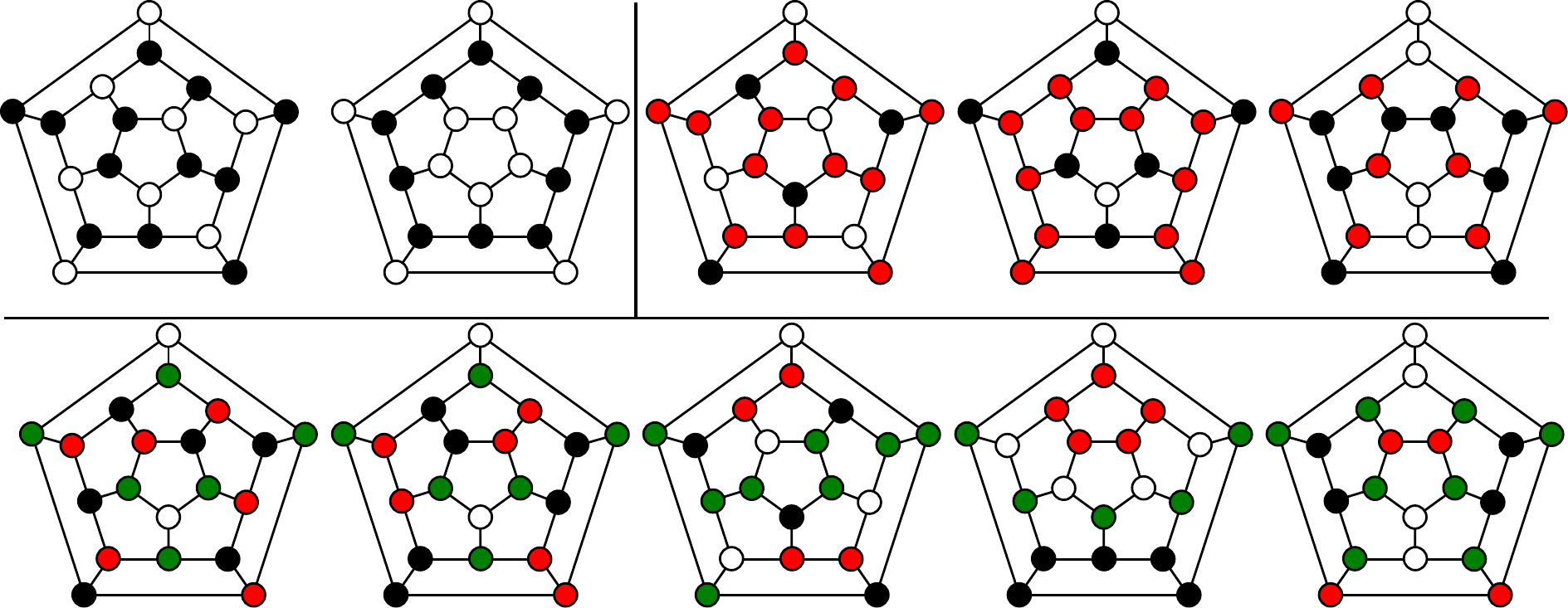} \]
\caption{The perfect 2-, 3- and 4-colorings of the
dodecahedral graph. \label{fig:dodecahedron}}
\end{figure}

\begin{figure}[h]
\[ \includegraphics[width=0.7\textwidth]{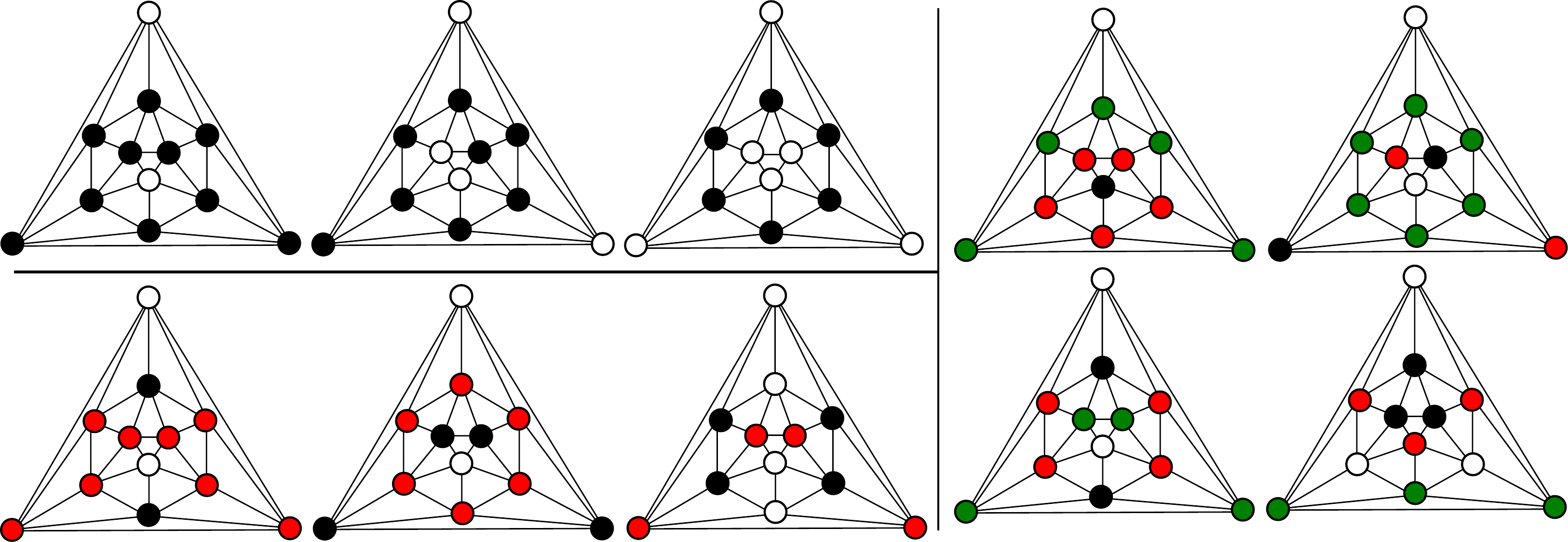} \]
\caption{The perfect 2-, 3- and 4-colorings of the
icosahedral graph. \label{fig:icosahedron}}
\end{figure}

\end{document}